\documentclass{amsart}

\usepackage{graphicx}
\usepackage{amssymb,latexsym}
\usepackage{amsmath,amsfonts,amstext,bbm}

\newtheorem{theorem}{Theorem}
\newtheorem{lemma}{Lemma}

\def\XXint#1#2#3{{\setbox0=\hbox{$#1{#2#3}{\int}$ }
\vcenter{\hbox{$#2#3$ }}\kern-.6\wd0}}

\begin{document}
\title[A QUANTITATIVE BALIAN-LOW THEOREM FOR HIGHER DIMENSIONS]
      {A quantitative Balian-Low theorem for higher dimensions}
      
\author{Faruk Temur}
\address{Department of Mathematics\\
        Izmir Institute of Technology  }
\email{faruktemur@iyte.edu.tr}
\date{January 15, 2016}

\begin{abstract}
We extend the quantitative Balian-Low theorem of Nitzan and Olsen to  higher dimensions. 
\end{abstract} 

\maketitle

\section{introduction}
Uncertainty principles are statements that limit  simultaneous concentration of   functions and their Fourier transforms. In the last two decades significant  attention has been paid to quantifying the maximum concentration that can be achieved. In this vein Nazarov proved in his seminal work \cite{na} that, for a  function $g\in L^2(\mathbb{R})$, and two  sets of finite measure $\mathcal{R},\mathcal{L}$  we have
\begin{equation*}
\int_{\mathbb{R}\setminus \mathcal{R}}|g(x)|^2dx + \int_{\mathbb{R}\setminus \mathcal{L}}|\widehat{g}(\xi)|^2d\xi \geq e^{-C|R||L|}\|g\|_{L^2(\mathbb{R})}^2
\end{equation*}    
 for an absolute constant $C>0$. This result quantifies the Heisenberg uncertainty priciple.
 Similarly it is possible to quantify the Balian-Low theorem \cite{bal,da,lo} which states that  if the Gabor system 
 \begin{equation*}
 G(g):=\{e^{2\pi i nx}g(x-m)\}_{(m,n)\in \mathbb{Z}^2}
 \end{equation*}
 generated by the function $g$ is a Riesz basis, we must have
 \begin{equation*}
 \int_{\mathbb{R}}|g(x)|^2x^2dx=\infty  \ \ \text{or}  \ \ \int_{\mathbb{R}}|\widehat{g}(\xi)|^2\xi^2d\xi=\infty.
 \end{equation*}
Nitzan and Olsen, \cite{no}, quantified this theorem by proving for $g$ as  above and $R,L$ are two real numbers with $R,L\geq 1$ 
 \begin{equation*}
 \int_{|x|\geq R}|g(x)|^2dx + \int_{|\xi| \geq L}|\widehat{g}(\xi)|^2d\xi \geq \frac{C}{RL}.
 \end{equation*}
where $C$ only depends on the Riesz basis bounds for the function $g$.

As is seen all these results are one-dimensional in nature. Although analogous results higher dimensions are conjectured, due to possibly much more complicated geometry of an arbitrary set in higher dimensions, progress has been more limited. One result in this direction is that of Jaming \cite{pj}  stating that for $g\in L^2(\mathbb{R}^d)$, and two  sets of finite measure $\mathcal{R},\mathcal{L}$
we have
\begin{equation*}
\int_{\mathbb{R}^d\setminus \mathcal{R}}|g(x)|^2dx + \int_{\mathbb{R}^d\setminus \mathcal{L}}|\widehat{g}(\xi)|^2d\xi \geq e^{-CD}\|g\|_{L^2(\mathbb{R}^d)}^2
\end{equation*} 
where $D=\min \{|\mathcal{R}||\mathcal{L}|,\alpha(\mathcal{R})|\mathcal{L}|^{\frac{1}{d}},\alpha(\mathcal{R})|\mathcal{L}|^{\frac{1}{d}}\}$ with $\alpha$ denoting the mean width of a set $S$ given by
\begin{equation*}
\alpha(S):= \int_{SO(d)} P_{\rho}(S)dv_d(\rho)
\end{equation*}
with $dv_d$ being the normalized Haar measure on the group of rotations  $SO(d)$, and $P_{\rho}(S)$ being the measure of the projection of $S$ on the line obtained by applying the rotation $\rho$ to the line spanned by the vector $(1,0,\ldots,0)$.  It is conjectured that $D$ can be replaced by $|\mathcal{R}|^{\frac{1}{d}}|\mathcal{L}|^{\frac{1}{d}}$. Thus, this result is essentially optimal if one of the sets $\mathcal{R},\mathcal{L}$ is very round, but it is far from optimal even when both sets are simple rectangles.
Our aim in this work is to   extend  to higher dimensions the work of Nitzan and Olsen \cite{no}, to investigate localization on rectangles.
\begin{theorem}
Let $g\in L^2(\mathbb{R}^d)$ be such that the Gabor system generated by $g$ 
 \begin{equation*}
  G(g):=\{e^{2\pi i nx}g(x-m)\}_{(m,n)\in \mathbb{Z}^{2d}}
  \end{equation*}
is a Riesz basis. Let $R_i,  L_i\geq 1$ be real numbers for each $  1\leq i \leq d$.  Let $\mathcal{R},\mathcal{L}$ be the d-dimensional rectangles $\mathcal{R}:=(-R_1,R_1)\times  \ldots  \times(-R_n,R_n)$, and $\mathcal{L}:=(-L_1,L_1)\times \ldots  \times(-L_n,L_n)$. We then have for a constant $C$ depending only on the Riesz basis bounds of $g$
\begin{equation*}
 \int_{\mathbb{R}^n \setminus \mathcal{R}}|g(x)|^2dx + \int_{\mathbb{R}^n \setminus  \mathcal{L}}|\widehat{g}(\xi)|^2d\xi \geq \frac{C}{R_i L_i}
 \end{equation*}
for any $ 1\leq i \leq d$. The theorem is sharp in the sense that the term ${C}/{R_i L_i}$ cannot be replaced by ${C}\log R_iL_i/{R_i L_i}$
\end{theorem} 
We observe that  the theorem allows us to choose the index $i$ that makes the right hand side largest. Since we must have $R_iL_i  
\leq |\mathcal{R}|^{\frac{1}{d}}|\mathcal{R}|^{\frac{1}{d}}$ at least for some values of $i$, the term 
 $C/R_iL_i$  can be replaced by $C/|\mathcal{R}|^{\frac{1}{d}}|\mathcal{L}|^{\frac{1}{d}}$ in the teorem.

The rest of the paper is organized as follows. We  introduce in the second section some standard definitions and results that will be used for the rest of the paper. Then we give  certain properties of quasiperiodic functions that Nitzan and Olsen uncovered in their work. In section 4 we use these properties to prove our estimate, and then discuss certain extensions of it. We will also,  
 using a function introduced in \cite{bcgp}, construct a function to show that our estimate is sharp. 


\section{Preliminaries}
In this section we will introduce concepts that will be used throughout the rest of the paper. Further information on all of these concepts can be found in \cite{gr}.  We start with Riesz bases. For a separable Hilbert space $H$, a system $\{v_n\}$ in $H$ is a Riesz basis 
  if it is complete in $H$ and
\begin{equation*}
    A \sum|a_n|^2 \leq \|\sum a_nv_n \|^2 \leq B \sum|a_n|^2
\end{equation*}
for any  sequence $\{a_n\}\in \ell^2$ and two positive constants  $A$ and $B$. The largest such $A$ and
smallest such $B$  are called the
Riesz basis bounds. An equivalent definition of 
a Riesz basis is that  it is the image of an orthonormal basis
under a bounded and invertible linear operator.

We now introduce the Zak transform, which is an extremely useful tool in the study of Gabor systems.
    Let $g \in L^1(\mathbb{R}^d)$. The Zak transform of $g$ is defined for  $(x,y) \in \mathbb{R}^{2d}$ as
    \begin{equation*}
        Zg(x,y) = \sum_{k \in \mathbb{Z}^d} g(x-k) e^{2\pi i k\cdot y}.
    \end{equation*}   
It is immediate from this definition and the Plancherel theorem   that  the Zak transform induces a unitary operator from $ L^2(\mathbb{R}^d)$ to $ L^2([0,1]^{2d})$. Thus for $g \in L^2(\mathbb{R}^d)$, the Zak transform   $Zg$ takes complex values for almost all $(x,y)\in \mathbb{R}^{2d}$. We let $e_1,e_2, \ldots, e_d$ be the canonical basis of $\mathbb{R}^d$.  For $g \in L^2(\mathbb{R}^d)$ and $1\leq i \leq d$ the function $Zg$  satisfies
\begin{equation}\label{qp}
    Zg(x,y+e_i) = Zg(x,y), \qquad \text{and} \qquad Zg(x+e_i,y) = e^{2\pi i  y_i} Zg(x,y).
\end{equation}
 We call this property quasiperiodicity. The Zak transform  relates to the Fourier transform as follows 
\begin{equation}\label{zf}
    Z\widehat{g}(x,y) = e^{2\pi i x\cdot y} Zg(-y,x).
\end{equation}
 Furthermore, we have for any Schwarz class function $\phi$,
\begin{equation}\label{p3} 
    Z(g \ast \phi) = Zg \ast_1 \phi,
\end{equation}
where for $Zg=Zg(x,y)$ notation $\ast_1$ means  convolution in the first variable $x$. With the Zak transform we can easily characterize the Gabor systems that are  Riesz bases: a Gabor system $G(g)$ is a Riesz basis if and only if 
\begin{equation} \label{zb}
    A \leq |Zg(x,y)|^2 \leq B, 
\end{equation}
where $A,B$ are Riesz basis bounds. This fact makes the Zak transform a   fundamental tool in the study of  the Gabor systems.

\section{Properties of quasiperiodic functions}

Nitzan and Olsen deduced their result by quantifying discontinuous  behovior of arguments of quasiperiodic functions. It is well known that  a branch of  the argument of a quasiperiodic function on $\mathbb{R}^2$ cannot be continuous. Nitzan and Olsen went further, and quantified this fact with the following lemma. For the sake of completeness we provide a proof here.

\begin{lemma}\label{lem1}
 Let $G$ be a complex valued quasiperiodic function on $\mathbb{R}^2$, and let $H$ be a branch of its argument, that is 
\[
G(x,y)=|G(x,y)|e^{2\pi i H(x,y)}.
\]
  Let $k,n\geq 8$  be two integers, and let $(x,y)\in [0,1/k)\times[0,1/n)$. Then there exist two integers $1\leq i< k, \ 1\leq j <n$ such that at least one of the following is true for every $m \in \mathbb{Z}$

\begin{equation*}
\begin{aligned}
&|H(x+{(i+1)}/{k},y+{j}/{n})-H(x+{i}/{k},y+{j}/{n})-m|    > 1/8,  \\
&|H(x+{i}/{k},y+{(j+1)}/{n})-H(x+{i}/{k},y+{j}/{n})-m|    >  1/8.
\end{aligned}
\end{equation*}

\end{lemma}

\begin{proof}
We assume to the contrary that there is a branch of the argument $H$ for which the claim  does not hold for a point $(x,y)\in [0,1/k)\times[0,1/n)$, with $k,n \geq 8$. We let for $i,j$ integers $h_{i,j}$ denote 
$H(x+i/k,y+j/n)$. We observe that if this $H$ presents a counterexample to the lemma, then so does infinitely many others, for by adding integers to $H$ at points  $(x+i/k,y+j/n)$ we obtain other counterexamples.
Since $H$ can be chosen from an infinite collection of counterexamples,   we can, to some extent, dictate the values $h_{i,j}$. Below we will do this to obtain a contradiction with the quasiperiodicity.

We fix $h_{0,0}$, and choose $h_{i,0}, \  1\leq i \leq  k$ so as to satisfy $|h_{i,0}-h_{i-1,0}|\leq 1/8$. Thus given $h_{0,0}$  fixed, we choose $h_{i,0}, \  1<i<k$ one by one, starting with $h_{1,0}$, so that their distance from the choice before is not more than $1/8$.  Now we have  $h_{i,0}, \  0 \leq i \leq  k$   are all fixed.  Using  $h_{i,0}, \  0 \leq i < k$  we choose $h_{i,j}, \  1 \leq j \leq  n$ so as to satisfy $|h_{i,j}-h_{i,j-1}|\leq 1/8$. Finally we choose $h_{k,j}$ for $1<j\leq n$. By quasiperiodicity we must have $h_{k,0}=h_{0,0}+y+l$ for some  integer $l$. We choose $h_{k,j}=h_{0,j}+y+j/n+l$ for $1<j\leq n$. Thus we have $|h_{k,j}-h_{k,j-1}|\leq 1/4$ for $1\leq j \leq n$. 

We claim that with these choices we also have $|h_{i,n}-h_{i-1,n}|\leq 1/8$ for $1\leq i \leq k$.
This we will prove through an iteration. We observe that since $H$ is assumed to be a counterexample to the lemma,  $|h_{i,1}-h_{i-1,1}-m_{i,1}|\leq 1/8$ for an integer $m_{i,1}$ for each $1\leq i \leq k$. But it is also clear from the construction of $H$, and the triangle inequality that we have $|h_{i,1}-h_{i-1,1}|\leq  1/2$. Therefore, $m_{i,1}=0$ for each value of $i$. If we apply the same reasoning we can obtain that for each $0\leq j< n$ we have $|h_{i,j}-h_{i-1,j}|\leq 1/8$. By quasiperiodicity  $|h_{i,n}-h_{i-1,n}-m_{i,n}|\leq  1/8$, for some integer $m_{i,n}$ for each $1\leq i \leq k$. But we have just discovered that $|h_{i,n-1}-h_{i-1,n-1}|\leq 1/8$ for each $1\leq i \leq k$. This, together with the triangle inequality, and the construction of  $H$  establishes the claim.
  
  We now  obtain the contradiction promised by calculating two sides of the obvious equality $(h_{k,n}-h_{k,0})-(h_{0,n}-h_{0,0})=(h_{k,n}-h_{0,n})-(h_{k,0}-h_{0,0})$ in two different ways.  By quasiperiodicity of $H$, for any $0\leq i \leq k$ the difference $h_{i,n}-h_{i,0}$ must be an integer.  But since we know that $|h_{i,n}-h_{i-1,n}|\leq 1/8$, and  $|h_{i,0}-h_{i-1,0}|\leq 1/8$    for each $0< i \leq k$  the integers $h_{i,n}-h_{i,0}$, $h_{i-1,n}-h_{i-1,0}$ must be the same. Thus $h_{k,n}-h_{k,0}$ and $h_{0,n}-h_{0,0}$ must be the same, hence $(h_{k,n}-h_{k,0})-(h_{0,n}-h_{0,0})$ must be zero.  On the other hand, from our construction of $H$ we have $(h_{k,j}-h_{0,j})-(h_{k,j-1}-h_{0,j-1})=1/n$ for each $1\leq j\leq n$. Thus $(h_{k,n}-h_{0,n})-(h_{k,0}-h_{0,0})=1$. A contradiction.

\end{proof}
The lemma we have just proved  suggests that the set of  points for which  a branch of the argument of a quasiperiodic function changes very quickly must have a measure at least $k^{-1}\cdot n^{-1}$. The next lemma makes this rigorous.

\begin{lemma}\label{lem2}
    
   Let $A>0$ be a  constant, and let $G$ be a complex valued quasiperiodic function on $\mathbb{R}^2$ with $|G|\geq A$. Then for any two integers   $k,n\geq 8$ we have a  set $S \subseteq [0,1]^2$ of measure
    at least $k^{-1}\cdot n^{-1}$, such that for all $(x,y) \in S$ we have
    \begin{equation*}
         |G  ( x+k^{-1}, y )-G  ( x, y )| \geq A/3 \ \ \ \text{or} \ \ \
        | G  ( x, y+n^{-1})- G  ( x, y)| \geq A/3.
    \end{equation*}
\end{lemma}

\begin{proof}
Let $H$ be a measurable branch of the argument of $G$. We will apply the Lemma 1 to this $H.$
Let $1\leq i<k, \ 1\leq j<n$, and let $m$ be an integer. Let $S'_{i,j,m,1}$ be the set of all $(x,y)\in [0,k^{-1})\times[0,n^{-1})$  for which the first inequality of the previous lemma holds. We similarly define $S'_{i,j,m,2}$. Clearly these sets are measurable. From these sets we define 
\[S'_{i,j,1}:=\bigcap_{m\in \mathbb{Z}}S_{i,j,m,1}, \ \ \   S'_{i,j,2}:=\bigcap_{m\in \mathbb{Z}}S_{i,j,m,2},  \]  
and we let $S'_{i,j}:=S'_{i,j,1}\cup S'_{i,j,2}$. Then from the previous lemma, we have the equality
\[ [0,k^{-1})\times[0,n^{-1})=  \bigcup_{i,j}S'_{i,j} \]
Thus, sum of measures of the sets on the right hand side is at least $k^{-1}\cdot n^{-1}$.  For each $i,j$ we define  $S_{i,j}$ to be  the translate of $S'_{i,j}$ by $(i,j)$. Thus the sets  $S_{i,j}$ are disjoint, and  for fixed $i,j$ the set $S_{i,j}$ has the  same measure as $S'_{i,j}$.  If  we define $S$ to be the union of all $S_{i,j}$, its measure is at least $k^{-1}\cdot n^{-1}$, and for an element $(x,y) \in S$ one of the following is true for all integers $m$
\begin{equation}\label{lem1}
         |H  ( x+k^{-1}, y )-H  ( x, y )-m| > 1/8, \  \ \ 
        | H  ( x, y+n^{-1})- H  ( x, y)-m| > 1/8.
   \end{equation}
Now suppose the first inequality is true for all $m$. We know that $|G(x,y)|,|G  ( x+k^{-1}, y )|\geq A$, but we do not know their exact relation to each other , and this prevents us from immediately concluding the proof. To circumvent this we proceed in two cases.  If 
 $ ||G  ( x+k^{-1}, y )| -|G  ( x, y )|| \geq A/3$, then
we have the  crude estimate
\[|G  ( x+k^{-1}, y )-G  ( x, y )|  \geq  ||G  ( x+k^{-1}, y )|-|G(x,y)|| \geq A/3.\]
If, on the other hand $ ||G  ( x+k^{-1}, y )| -|G  ( x, y )|| < A/3$, then by adding and subtracting the same term we can write $|G  ( x+k^{-1}, y )-G  ( x, y )|$  as
\[
 |[|G  ( x+k^{-1}, y )|-|G  ( x, y )|]e^{2\pi i  H( x+k^{-1}, y )}+|G  ( x, y )|[e^{2\pi i  H( x+k^{-1}, y )}-e^{2\pi i  H( x, y )}]|.
\]
This, by triangle inequality, cannot be less than
\[
 |G  ( x, y )||e^{2\pi i  H( x+k^{-1}, y )}-e^{2\pi i  H( x, y )}|-||G  ( x+k^{-1}, y )|-|G  ( x, y )||
\]
We observe that  
\begin{equation*} 
 \begin{aligned}
 |e^{2\pi i  H( x+k^{-1}, y )}-e^{2\pi i  H( x, y )}|&= |e^{2\pi i  H( x, y )}[ e^{2\pi i [H( x+k^{-1}, y )- H( x, y )]}-1]|\\ &=| e^{2\pi i [H( x+k^{-1}, y )- H( x, y )]}-1|
  \end{aligned}
\end{equation*} 
We know that the distance of $H( x+k^{-1}, y )- H( x, y )$ to any integer is more than $1/8$, which means that the last term is more than $2/3$. Thus returning with this information back to our estimate 
  \[|G  ( x, y )||e^{2\pi i  H( x+k^{-1}, y )}-e^{2\pi i  H( x, y )}|-||G  ( x+k^{-1}, y )|-|G  ( x, y )|| \geq A/3.\]
  Thus in any  case if the first inequality in \eqref{lem1}
holds for all integers, we have $|G  ( x+k^{-1}, y )-G  ( x, y )|\geq A/3$. Similarly if the second inequality  in \eqref{lem1}
holds for all integers, we have $|G  ( x, y+n^{-1} )-G  ( x, y )|\geq A/3$, and this concludes the proof.
\end{proof}


\section{Proof of the main result}
We shall start with a lemma that will be the fundamental tool in proving our theorem. The last lemma tells us that given a quasiperiodic function there is a set of certain size near which  the function changes rapidly. Therefore on this set the function must also differ from its average over balls of large enough size. The next lemma makes rigorous this idea, using convolutions with Schwartz class functions instead of averages over balls. 

\begin{lemma}\label{lem3}
    Let $A,B >0$, and let $1\leq i \leq d$.  Given two Schwartz  functions $\phi,\psi$ on $\mathbb{R}^d$,
    and any $g\in L^2(\mathbb{R}^d)$
    with
    $A\leq |Zg|\leq B$ almost everywhere, there is 
 a set $S_i\subseteq [0,1]^{2d}$ of measure 
    at least 
   ${A^2
   /4000 B^2(1+\|\phi_{i}\|_1)(1+\|\psi_{i}\|_1)}$ with $\phi_i,\psi_i$ denoting ith partial derivatives of $\phi,\psi$,   
   such that for all $(x,y) \in S_i$ we have
    \begin{equation*}
       |Zg(x,y) - Z(g \ast \phi)(x,y)|  \geq  A/12  \ \ \ or \ \ \
     |Z\hat{g}(x,y) - Z(\hat{g} \ast \psi)(x,y)| \geq  A/12.
    \end{equation*}
\end{lemma}

\begin{proof}
 For any $k_i>0$, and any $1\leq i \leq d$ from the property  \eqref{p3} of the Zak transform 
 $|Z(g \ast \phi)(x+k_i^{-1}\cdot e_i,y) - Z (g \ast \phi)(x,y) |$ can be written as 
     $|Zg \ast_1 \phi(x+k_i^{-1}\cdot e_i,y) - Z g \ast_1 \phi(x,y) | $. 
Then we have
\begin{equation*}
 \begin{aligned}
 &\leq \int_{\mathbb{R}^d}|Zg(u,y)||\phi(x-u+k_i^{-1}\cdot e_i)-\phi(x-u)|du\\
 &\leq B\cdot \int_{\mathbb{R}^d}|\phi(x-u+k_i^{-1}\cdot e_i)-\phi(x-u)|du.
 \end{aligned}
\end{equation*}
 Let $u_i\in \mathbb{R}$ denote $i$th coordinate of $u$ and let $\overline{u}\in \mathbb{R}^{d-1} $ with $\overline{u} $ being obtained from $u$ by removing $u_i$.   Since $\phi$ is a Schwartz function we can write, 
\begin{equation*}
 \begin{aligned}
  &= B\cdot \int_{\mathbb{R}^{d-1}}\int_{\mathbb{R}}|\phi(x-u+k_i^{-1}\cdot e_i)-\phi(x-u)|du_id\overline{u}\\
 &= B\cdot \int_{\mathbb{R}^{d-1}}\int_{\mathbb{R}}\int_{0}^{k_i^{-1}}|\phi_{i}(x-u+v_i\cdot e_i)|dv_idu_id\overline{u}\\
 &= B\cdot \int_{\mathbb{R}^{d-1}}\int_{0}^{k_i^{-1}}\int_{\mathbb{R}}|\phi_{i}(x-u+v_i\cdot e_i)|du_idv_id\overline{u}.
 \end{aligned}
\end{equation*}
 Observing that the inner integral is independent of $v_i$  we can write
\begin{equation*}
 \begin{aligned}
 &= B\cdot\int_{\mathbb{R}^{d-1}}\int_{0}^{k_i^{-1}}\int_{\mathbb{R}}|\phi_{i}(x-u)|du_idv_id\overline{u}\\
 &={B}\cdot{k_i}^{-1}\cdot\int_{\mathbb{R}^{d-1}}\int_{\mathbb{R}}|\phi_{i}(x-u)|du_id\overline{u}, 
 \end{aligned}
\end{equation*}
and obviously this last term is ${B}\cdot{k_i}^{-1}\cdot\|\phi_{i}\|_1$. Since we have the property \eqref{zf}, we can apply   the same process to   obtain for any $n_i>0$, and any $(x,y)\in \mathbb{R}^{2d}$ 
\begin{equation}\label{eqphi}
    \begin{aligned}
    |Z(\widehat{g} \ast \psi)(y+{n_i}^{-1}\cdot e_i,-x) - Z (\widehat{g} \ast \psi)(y,-x) |\leq {B}\cdot{n_i}^{-1}\cdot \|\psi_{i}\|_1.  
\end{aligned}
\end{equation}
 If we  
choose  $k_i,n_i\geq 8$  to be the smallest integers that satisfy
\begin{equation*}
 k_i \geq  \frac{8B}{A}(1+\|\phi_{i}\|_1) \qquad \text{and} \qquad n_i \geq  \frac{24\pi B}{A}(1+\|\psi_{i}\|_1),
\end{equation*}
we have
\begin{equation*}
    \begin{aligned}
     |Z(g \ast \phi)(x+{k}_i^{-1}\cdot e_i,y) - Z (g \ast \phi)(x,y) | &\leq  {A}/{8}\\ |Z(\widehat{g} \ast \psi)(y+{n_i}^{-1}\cdot e_i,-x) - Z (\widehat{g} \ast \psi)(y,-x) | &\leq  {A}/{72}  
\end{aligned}
\end{equation*}

We will use the properties of quasiperiodic functions derived in the last section  to obtain another estimate, which, combined with the last two will suffice to complete the proof. To this end we introduce a slight modification of  $Zg(x,y)$ as follows. Let   $G(x,y):=Zg(x,y)$ when $A \leq |Zg(x,y)|\leq B$, and when this is not the case, let $G(x,y)=B$ for  $(x,y) \in [0,1)^{2d}$, and  extend  it  so that it will be quasiperiodic. This $G$ is a measurable,  complex valued, quasiperiodic function , and $A \leq |G(x,y)|\leq B$ everywhere. Now we define 
$G_{\overline{x},\overline{y}}(x_i,y_i)=G(x,y)$ with $(\overline{x},\overline{y})$  denoting an element of $\mathbb{R}^{2d-2}$ obtained by removing $x_i,y_i$ from $(x,y)\in \mathbb{R}^{2d}$. To this $G_{\overline{x},\overline{y}}(x_i,y_i)$  we wish to apply the Lemma 2. We see that for any $(\overline{x},\overline{y})$ by definition  it is complex valued,  quasiperiodic, and satisfies $A \leq |G_{\overline{x},\overline{y}}(x,y)|\leq B$ for all $(x_i,y_i)$. Also by applying  Fubini-Tonelli theorem for complete measures, see Theorem 2.39 in \cite{fo}, or Theorem 8.12 in \cite{ru}, to $G\chi_{[0,1)^{2d}}$ we see that for almost all  $(\overline{x},\overline{y})\in \mathbb{R}^{2d-2}$ the function $G_{\overline{x},\overline{y}}\chi_{[0,1)^{2}}$  is measurable,  and hence by quasiperiodicity for almost all  $(\overline{x},\overline{y})\in \mathbb{R}^{2d-2}$ the function $G_{\overline{x},\overline{y}}(x,y)$  is measurable. Thus we can apply Lemma 2 to this function for almost all    $(\overline{x},\overline{y})\in [0,1)^{2d-2}$. 
Let $S_{\overline{x},\overline{y},i}$ be the set described in Lemma 2 for such a point  $(\overline{x},\overline{y})$ with $k_i,n_i$ as chosen above. Then for $(x_i,y_i)\in S_{\overline{x},\overline{y},i}$ we have
  \begin{equation*}
          |G_{\overline{x},\overline{y}}  ( x_i+k_i^{-1}, y_i )-G_{\overline{x},\overline{y}}  ( x_i, y_i )|  \geq A/3 \ \ \text{or} \ \
         | G_{\overline{x},\overline{y}}  ( x_i, y_i+n_i^{-1})- G_{\overline{x},\overline{y}}  ( x_i, y_i)| \geq A/3.
     \end{equation*}
We let $S''_i$ to be the set of $(x,y)\in [0,1)^{2d}$, such that $(x_i,y_i)\in S_{\overline{x},\overline{y},i}$.  This set has measure at least $k_i^{-1}\cdot n_i^{-1}$, and for elements of this set we have
   \begin{equation*}
         |G(x+k_i^{-1}\cdot e_i,y) - G (x,y)| \geq  A/3 \ \ \text{or} \ \
       |G(x,y+n_i^{-1}\cdot e_i) - G (x,y)| \geq  A/3.
      \end{equation*}
Since the set $F$ of points for which $Zg\neq G$ has measure zero if we remove from  $S''_i$ the set $F$ and  its translations by $-k_i^{-1}\cdot e_i$ and, $-n_i^{-1}\cdot e_i$, the remainder has the same measure as $S''_i$ and  for $(x,y)$ in this  remainder, which we will denote by $S'_i$, we have
   \begin{equation*}
         |Zg(x+k_i^{-1}\cdot e_i,y) - Zg (x,y)| \geq  A/3 \ \ \text{or} \ \
       |Zg(x,y+n_i^{-1}\cdot e_i) - Zg (x,y)| \geq  A/3.
      \end{equation*}

 We let $S_i'=U_i'\cup V_i'$ where for elements of  $U_i'$  the first of these inequalities holds, and for elements of  $V_i'$  the second one holds. Then for $(x,y)\in U_i'$ we have 

\begin{equation*}
      \begin{aligned}
       &|Zg(x,y) - Z(g \ast \phi)(x,y)| + |Zg(x+k_i^{-1}\cdot e_i,y) - Z(g \ast \phi)(x+k_i^{-1}\cdot e_i,y)|
     \\
       \geq &|Zg(x+k_i^{-1}\cdot e_i,y) -Zg(x,y)  -Z(g \ast \phi)(x+k_i^{-1}\cdot e_i,y)+ Z(g \ast \phi)(x,y)  |
     \\
     \geq & |Zg(x+k_i^{-1}\cdot e_i,y) - Zg (x,y)| - |Z(g \ast \phi)(x+k_i^{-1}\cdot e_i,y) - Z (g \ast \phi)(x,y) |
    \\
        \geq & A/6
    \end{aligned}
    \end{equation*}
Therefore one of the following is certainly true for any element $(x,y)\in U_i'$
\begin{equation*}
      \begin{aligned}
       |Zg(x,y) - Z(g \ast \phi)(x,y)| &\geq A/12 \\  |Zg(x+k_i^{-1}\cdot e_i,y) - Z(g \ast \phi)(x+k_i^{-1}\cdot e_i,y)|   &\geq A/12.   
    \end{aligned}
    \end{equation*}
 We thus have $ U_i'= U_{i_1}' \cup U_{i_2}'$ with the element $(x,y)$  belonging to $U_{i_1}'$ set if the first inequality holds, and to  $U_{i_2}'$ if the second holds, and to both if both inequalities hold. Obviously at least one of these sets have a measure not less than  half of the measure of $U_i'$. Thus if  we consider  the union of $U_{i_1}'$ and the translate of $U_{i_2}'$ by $(k_i^{-1}\cdot e_i,0)$, then its measure is not less than half the measure of $U_i'$. But it may be that some of elements of this union are not in $[0,1]^{2d}$ due to the translation. Therefore we define $U_i$ to be the union of the set of all elements in $[0,1]^{2d}$ and the set of all elements outside $[0,1]^{2d}$ translated by $(-e_i,0)$. This $U_i$ then has at least a quarter of the measure of $U_i'.$   And any element of this set satisfy the first inequality of our lemma.
 
 We now turn to $(x,y)\in V_i'$. We have the equation \eqref{eqphi}. We also have
 \begin{equation*}
       \begin{aligned}
        &|Z\widehat{g}(y+n_i^{-1}\cdot e_i,-x) - Z\widehat{g} (y,-x)|\\
         \geq  &| e^{-2\pi i (x\cdot y+n_i^{-1}x_i)}Z{g}(x,y+n_i^{-1}\cdot e_i) - e^{-2\pi i x\cdot y}Z{g} (x,y)| \\ 
        \geq  &| e^{-2\pi i n_i^{-1}x_i}Z{g}(x,y+n_i^{-1}\cdot e_i) - Z{g} (x,y)| \\ \geq
            &| e^{-2\pi i n_i^{-1}x_i}[Z{g}(x,y+n_i^{-1}\cdot e_i)- Z{g}(x,y)] -Z{g}(x,y)[1-e^{-2\pi i n_i^{-1}x_i}]|  
  \\ \geq
              &| Z{g}(x,y+n_i^{-1}\cdot e_i)- Z{g}(x,y)| -|Z{g}(x,y)[1-e^{-2\pi i n_i^{-1}x_i}]| 
  \\ \geq
                 &| Z{g}(x,y+n_i^{-1}\cdot e_i)- Z{g}(x,y)| -B\cdot|1-e^{-2\pi i n_i^{-1}x_i}| 
  \end{aligned}
     \end{equation*}
 We can easily estimate, using the unit circle, that $|1-e^{-2\pi i n_i^{-1}x_i}| \leq A/12B$, thus the last term is not less than $A/4.$
 Combining this with \eqref{eqphi}, as we did in the case of elements of $U_i'$, we have for $(x,y)\in V_i'$ one of the following certainly true
 \begin{equation*}
       \begin{aligned}
       |Z\widehat{g}(y,-x) - Z(\widehat{g} \ast \psi)(y,-x)| &\geq A/12 \\  |Z\widehat{g}(y+n_i^{-1}\cdot e_i,-x) - Z(\widehat{g} \ast \psi)(y+n_i^{-1}\cdot e_i,-x)|   &\geq A/12. 
     \end{aligned}
     \end{equation*}
 We let $ V_i'= V_{i_1}' \cup V_{i_2}'$ as before, and  obviously at least one of these sets have a measure not less than  half  the measure of $V_i'$. Thus if  we consider  the union of $V_{i_1}'$ and the translate of $V_{i_2}'$ by $(0,n_i^{-1}\cdot e_i)$, then its measure is not less than half the measure of $V_i'$. But it may be that some of the elements of this set are not in $[0,1]^{2d}$ due to the traslation applied. We therefore take  the union of elements in $[0,1]^{2d}$ with $(0,-e_i)$ translates of those that are not in $[0,1]^{2d}$, and if  we set  $V_{i}$ to be  the set of points  $(y,-x)$ such that $(x,y)$ is in this last union, its measure is not less than a quarter of that of $V_i'$, it lies entirely in $[0,1]^{2d}$ and     any element of it satisfy the second  inequality of our lemma.   We finally define $S_i=U_i\cup V_i$ and easily observe that it  satisfies all of the required   properties.
\end{proof}

We will use what we learned from the study of the Zak transform of Riesz basis generators to prove Theorem 1.
Let $g $ be a function as in the theorem. 
We pick  a Schwarz class function $\rho$ on $\mathbb{R}^d$ such that $\widehat{\rho}$ is radially  symmetric, and satisfies
$|\widehat{\rho}|\leq 1$ everywhere, and
\begin{equation*}
     \widehat{\rho}(\xi) = \left\{ \begin{split}  1 & \quad \text{if} \quad |\xi| \leq 1,
     \\
      0 & \quad \text{if} \quad |\xi| \geq 2. \end{split} \right.
\end{equation*}
We define two Schwarz functions $\phi,\psi$ by
\begin{equation*}
   \begin{aligned}
    \phi(x_1,x_2,\ldots,x_d)&: = R_1 R_2\ldots R_d\cdot \rho(R_1x_1,R_2x_2,\ldots,R_dx_d)\\
     \psi(x_1,x_2,\ldots,x_d)&: = L_1 L_2\ldots L_d\cdot \rho(L_1x_1,L_2x_2,\ldots,L_dx_d).
    \end{aligned}
\end{equation*}
Therefore we have $\|\phi_{i}\|_1=R_i\|\rho_{i}\|_1$, and $\|\psi_{i}\|_1=L_i\|\rho_{i}\|_1$. Since $\rho$ is a fixed radial Schwarz function, $\|\rho_{i}\|_1$ is a fixed constant for every $i$, which we will denote by $\Gamma$. We apply Lemma 3 to obtain for any chosen $1\leq i \leq d$  a
set $S_i\subseteq [0,1]^{2d}$ with measure at least ${A
   /4000 B(1+R_i\Gamma)(1+L_i\Gamma)}$ such
that all $(x,y)\in S_i$ satisfy
\begin{equation*}
    A/144 \leq  |Z\widehat{g}(x,y) - Z(\widehat{g} \ast \phi)(x,y) |^2 +  |Zg(x,y) - Z(g \ast \psi)(x,y)|^2.
\end{equation*}
Since we assumed in the theorem $R_i,L_i \geq 1$ for every $1\leq i \leq d$, we have  $A
   /4000 B(1+R_i\Gamma)(1+L_i\Gamma)\geq  A
      /2^410^3 B \Gamma^2 R_i L_i$. Thus if we integrate over $S_i$
\begin{equation*}
\begin{aligned}
           A^2/10^8 B \Gamma^2 R_i L_i &\leq   \||Z\widehat{g} - Z(\widehat{g}\ast \phi)|^2  + |Zg - Z({g} \ast \psi)|^2 \|_{L^1([0,1]^{2d})}   \\
     &\leq   \|Z\widehat{g} - Z(\widehat{g}\ast \phi) \|^2_{L^2([0,1]^{2d})} + \|Zg - Z({g} \ast \psi)\|^2_{L^2([0,1]^{2d})} \\
&\leq   \|Z[\widehat{g} - (\widehat{g}\ast \phi)] \|^2_{L^2([0,1]^{2d})} + \|Z[g -({g} \ast \psi)]\|^2_{L^2([0,1]^{2d})}. 
\end{aligned}
\end{equation*}
 As the Zak transform is a unitary operator from $L^2([0,1]^{2d})$ to
$L^2(\mathbb{R}^d)$, we have
    \begin{equation*}
    =   \|\widehat{g} - (\widehat{g}\ast \phi) \|^2_{L^2(\mathbb{R}^{d})} + \|g -({g} \ast \psi)\|^2_{L^2(\mathbb{R}^{d})}.  
    \end{equation*}
 We apply the Plancherel theorem, and use the assumption that $\widehat{\rho}$ is radially symmetric to obtain      
    \begin{equation*}
    \begin{aligned}
     & \leq     \|g (1- \widehat{\phi}) \|^2_{L^2(\mathbb{R}^{d})} + \|\widehat{g} (1 -\widehat{\psi})\|^2_{L^2(\mathbb{R}^{d})}\\
    &\leq \int_{\mathbb{R}^d \setminus \mathcal{R}}|g(x)|^2dx + \int_{\mathbb{R}^d \setminus  \mathcal{L}}|\widehat{g}(\xi)|^2d\xi 
  \end{aligned}
  \end{equation*}
which proves our theorem with the constant $C$ in the theorem being not more than $A^2/10^8 B \Gamma^2$.

Our theorem can be extended without much effort in two  different directions. The first is to take the rectangles $\mathcal{R},\mathcal{L}$ directed along not the canonical basis but a different orthonormal basis. As long as we take both rectangles directed along the same orthonormal basis, our theorem generalizes easily by employing rotations. The second extension is to
more general Gabor systems that are produced by simple scaling of the canonical lattice: for $a,b$ real numbers we define
\begin{equation*}
	G(g,a,b):= \{ e^{2\pi
		i b  n x} g(x - ma)\}_{(m,n) \in \mathbb{Z}^{2d}}. 
\end{equation*}
It is possible for such a system to be a Riesz basis if $ab=1$.  Our result also holds for these  more general systems, and this can be seen by employing appropriate dilations. 

We now present the  counterexample showing that our estimate is sharp. Nitzan and Olsen observed that a function $f\in L^2(\mathbb{R})$ constructed in \cite{bcgp} satisfies $|Zf|=1$ on all of  $\mathbb{R}^2$,  and for $R,L\geq 1$
\begin{equation*}
    \int_{|x|\geq R} |{f(x)}|^2 d x + \int_{|\xi| \geq L}   |\widehat{f}(\xi)|^2 d \xi \leq \frac{1}{R^2} + \frac{\log L}{L^2}.
\end{equation*}
Since the Zak transform is unitary we have $\|f\|_{L^2(\mathbb{R})}=1$, and hence from the Plancherel theorem we further have  $\|\widehat{f}\|_{L^2(\mathbb{R})}=1$. This function is a counterexample showing that the result of Nitzan and Olsen, which is the $d=1$ case of our result, is sharp. We will construct a counterexample from this function to show that our result cannot be improved in any dimension $d$.

We let $x$ denote $(x_1,x_2,\ldots,x_d)$ and define  on $\mathbb{R}^d$ the function $g\in L^2(\mathbb{R}^d)$ by 
$g(x):= f(x_1)f(x_2)\ldots f(x_d)$.  Owing to this relation  $\|g\|_{L^2(\mathbb{R}^d)}=\|\widehat{g}\|_{L^2(\mathbb{R}^d)}=1$. We then have the same relation between the Fourier transforms of $f$ and $g$:
 $\widehat{g}(\xi)= \widehat{f}(\xi_1)\widehat{f}(\xi_2)\ldots \widehat{f}(\xi_d),$
  and  between the Zak transforms we have  
 $Zg(x,y)=Zf(x_1,y_1)Zf(x_2,y_2)\ldots Z(x_d,y_d).$
  Therefore $|Zg|=1$ everywhere on $\mathbb{R}^{2d}$, and this means that $g$ generates a Gabor frame, and satisfies the hypothesis of our theorem.  On the other hand
   for rectangles $\mathcal{R}=(-R_1,R_1)\times \ldots \times (-R_d,R_d)$ and $\mathcal{L}=(-L_1,L_1)\times \ldots \times (-L_d,L_d)$ observe that by the  definition of $g$ we have
   \begin{equation*}
       \int_{|x_i|\geq R_i } |g(x)|^2 dx=  \int_{|x_i|\geq R_i } |f(x_i)|^2 dx_i, \ \ \ \  \int_{|\xi_i|\geq L_i } |\widehat{g}(\xi)|^2 d\xi=  \int_{|\xi_i|\geq L_i } |\widehat{f}(\xi_i)|^2 d\xi_i
   \end{equation*}
   for any index $i$. Therefore
   \begin{equation*}
       \begin{aligned}
       \int_{\mathbb{R}^{d}\setminus \mathcal{R} } |g(x)|^2 d x + \int_{\mathbb{R}^{d}\setminus \mathcal{L} }   
       |\widehat{g}(\xi)|^2 d \xi 
 &\leq \sum_{i=1}^d  \int_{|x_i|\geq R_i } |g(x)|^2 d x + \int_{|\xi_i|\geq L_i }  
     |\widehat{g}(\xi)|^2 d \xi  \\
    & = \sum_{i=1}^d  \int_{|x_i|\geq R_i } |f(x_i)|^2 d x_i + \int_{|\xi_i|\geq L_i }  
                            |\widehat{f}(\xi_i)|^2 d \xi_i \\
& = \sum_{i=1}^d     \frac{1}{R_i^2} + \frac{\log L_i}{L_i^2}.                       
       \end{aligned}
   \end{equation*}
  If we pick $R_1=R_2=\ldots =R_d=R$ and $L_1=L_2=\ldots =L_d=L$, and $L=R\log^{1/2} R$ we obtain
  \[\leq d\cdot \Big( \frac{1}{R^2} + \frac{\log ( R \log^{1/2} R)}{R^2\log R}\Big)\leq 3d\cdot \frac{1}{R^2}, \] 
whereas if we could improve right hand side of our estimate as mentioned we would, with such choices of $R_i,L_i$,  have
\[\int_{\mathbb{R}^{d}\setminus \mathcal{R} } |g(x)|^2 d x + \int_{\mathbb{R}^{d}\setminus \mathcal{L} }   
       |\widehat{g}(\xi)|^2 d \xi \geq C\cdot \frac{\log (R^2\log^{1/2}R)}{R^2\log^{1/2} R}\geq C \cdot \frac{\log^{1/2} R}{R^2},\]
with a constant $C$ independent of $R$, which  is a  clear contradiction

\end{document}